\documentclass[12pt]{article}
\usepackage{amssymb}
\usepackage{amsthm}
\usepackage{amsmath}
\usepackage{graphicx}
\usepackage{enumerate}
\usepackage{pict2e}
\usepackage{framed}

\DeclareMathOperator{\Conv}{Conv}
\DeclareMathOperator{\Convs}{Conv*}

\DeclareMathOperator{\Cl}{Cl}
\DeclareMathOperator{\BCl}{\mathbf{Cl}}

\DeclareMathOperator{\BConv}{\mathbf{Conv}}
\DeclareMathOperator{\BConvs}{\mathbf{Conv*}}

\DeclareMathOperator{\Max}{Max}
\DeclareMathOperator{\Min}{Min}

\newtheorem{theorem}{Theorem}[section]
\newtheorem{definition}[theorem]{Definition}
\newtheorem{lemma}[theorem]{Lemma}
\newtheorem{proposition}[theorem]{Proposition}

\newtheorem{remark}[theorem]{Remark}
\newtheorem{example}[theorem]{Example}
\newtheorem{corollary}[theorem]{Corollary}
\title{Operators on complemented posets}
\author{Michal~Botur, Ivan~Chajda and Helmut~L\"anger}
\date{}

\begin{document}

\footnotetext{Support of the research of the first and second author by the Czech Science Foundation (GA\v CR), project 24-14386L, entitled ``Representation of algebraic semantics for substructural logics'', support of the second author by IGA, project P\v rF~2025~008, and support of the research of the third author by the Austrian Science Fund (FWF), project 10.55776/PIN5424624, is gratefully acknowledged.}

\maketitle
	
\begin{abstract}
Given a complemented poset $\mathbf P=(P,\le,0,1)$, we can assign to every $x\in P$ the set $x^+$ of all its complements. We study properties of the operator $^+$ on $P$, in particular, we are interested in the case when $x^+$ forms an antichain or when $^+$ is involutive or antitone. We apply $^+$ to the set $\Min U(x,y)$ of all minimal elements of the upper cone $U(x,y)$ of $x,y$ and to the set $\Max L(x,y)$ of all maximal elements of the lower cone $L(x,y)$ of $x,y$. By using $^+$ we define four binary operators on $P$ and investigate their properties that are close to adjointness. We present an example of a uniquely complemented poset that is not Boolean. In the last section we study the orthogonality relation induced by complementation. We characterize when two elements of the Dedekind-MacNeille completion of $\mathbf P$ are orthogonal to each other. Finally, we extend the orthogonality relation from elements to subsets and we prove that two non-empty subsets of $P$ are orthogonal to each other if and only if their convex hulls are orthogonal to each other within the poset of all non-empty convex subsets of $P$.
\end{abstract}



\section{Introduction}

An element $a$ of a complemented lattice $\mathbf L=(L,\vee,\wedge,0,1)$ may have more than one complement if $\mathbf L$ is not distributive. However, R.~P.~Dilworth \cite D showed that in a distributive complemented (i.e.\ Boolean) lattice the complementation is unique. An element of a modular lattice may have several complements and it may happen that for an appropriate choice of these complements the lattice in question becomes orthomodular and for another choice it need not become orthomodular. One can overcome the problem of choosing an appropriate complement by considering for each $a\in L$ the set $a^+$ of all complements of $a$ without preferring any single of them. Hence, $^+$ is in fact a unary operator on $\mathbf L$, see e.g.\ \cite B. For some reasons concerning so-called unsharp logics, we are interested in complemented posets instead of lattices, see e.g.\ \cite{CL25b}. In posets we usually work with the sets $\Min U(x,y)$ and $\Max L(x,y)$ of all minimal elements of the upper cone $U(x,y)$ of the elements $x$ and $y$ and of all maximal elements of the lower cone $L(x,y)$ of $x$ and $y$, respectively. These sets substitute the missing elements $x\vee y$ and $x\wedge y$ in the case when the supremum or infimum of $x$ and $y$, respectively, do not exist. Of course, in the case when $x\vee y$ exists it coincides with $\Min U(x,y)$ and a similar situation occurs for $x\wedge y$. The operators $\Min U$ and $\Max L$ serve as logical connectives disjunction and conjunction, respectively, in certain kinds of unsharp logics, see again \cite{CL25b}.

Now it is a question whether also for a complemented poset $\mathbf P=(P,\le,0,1)$ it makes sense to define and study the operator $^+$ assigning to every element $a$ of $P$ the set of all its complements, and moreover, whether this operator shares similar properties with the corresponding operator in lattices. It makes no sense to study $^+$ for Boolean posets since in this case every element $a$ has a unique complement $a'$, thus $a^+=\{a'\}$ like in Boolean lattices. However, we show that a finite uniquely complemented poset need not be distributive and hence Boolean. We demonstrate the properties of $^+$ in several examples of complemented posets being neither Boolean nor lattices. Finally, using this operator we define other derived operators having similar properties like $t$-norms and we investigate corresponding adjointness properties.

\section{Basic concepts}

Basic concepts on posets are included e.g.\ in \cite B. Let us recall some well-known concepts and let us define some additional ones.

Let $\mathbf P=(P,\le)$ be a poset, $a,b\in P$, $A,B\subseteq P$ and $\,'$ a unary operation on $P$. Then {\em $\Max A$} and {\em $\Min A$} denote the set of all maximal and minimal elements of $A$, respectively. We define
\begin{align*}
 A\le B & \text{ if }x\le y\text{ for all }x\in A\text{ and all }y\in B, \\
   L(A) & :=\{x\in P\mid x\le A\}, \\
   U(A) & :=\{x\in P\mid A\le x\}, \\
 A\le B & \text{ if }x\le y\text{ for all }x\in A\text{ and }y\in B, \\
A\le_1B & \text{ if for every }x\in A\text{ there exists some }y\in B\text{ with }x\le y, \\
A\le_2B & \text{ if for every }y\in B\text{ there exists some }x\in A\text{ with }x\le y.
\end{align*}
Hence $(U,L)$ forms the Galois correspondence between the posets $(2^P,\subseteq)$ and $(2^P,\subseteq)$ induced by the binary relation $\le$ on $P$.

Here and in the following we identify singletons with their unique element. Instead of $L(\{a\})$, $L(\{a,b\})$, $L(A\cup\{a\})$, $L(A\cup B)$ and $L\big(U(A)\big)$ we simply write $L(a)$, $L(a,b)$, $L(A,a)$, $L(A,B)$ and $LU(A)$. Analogously we proceed in similar cases. 

The subset $A$ of $P$ is called an
\begin{itemize}
\item {\em upset} if $x\in A$, $y\in P$ and $x\le y$ imply $y\in A$,
\item {\em down set} if $x\in A$, $y\in P$ and $y\le x$ imply $y\in A$.
\end{itemize}
The smallest upset containing $a$ is $\uparrow a:=\{x\in P\mid a\le x\}$ and the smallest down set containing $a$ is $\downarrow a:=\{x\in P\mid x\le a\}$. The smallest upset including $A$ is
\[
A^\uparrow:=\bigcup_{x\in A}(\uparrow x)=\{y\in P\mid\text{there exists some }x\in A\text{ with }x\le y\}
\]
and the smallest down set including $A$ is
\[
A^\downarrow:=\bigcup_{x\in A}(\downarrow x)=\{y\in P\mid\text{there exists some }x\in A\text{ with }y\le x\}.
\]
The mappings $A\mapsto A^\uparrow$ and $A\mapsto A^\downarrow$ are closure operators on the poset $(2^P,\subseteq)$, i.e.\ we have
\begin{itemize}
\item $A\subseteq A^\uparrow$; $A\subseteq B\Rightarrow A^\uparrow\subseteq B^\uparrow$; $(A^\uparrow)^\uparrow=A^\uparrow$,
\item $A\subseteq A^\downarrow$; $A\subseteq B\Rightarrow A^\downarrow\subseteq B^\downarrow$; $(A^\downarrow)^\downarrow=A^\downarrow$.
\end{itemize}
Moreover, it is easy to see
\begin{itemize}
\item $A\le_1B\Leftrightarrow A\subseteq B^\downarrow\Leftrightarrow A^\downarrow\subseteq B^\downarrow$,
\item $A\le_2B\Leftrightarrow B\subseteq A^\uparrow\Leftrightarrow B^\uparrow\subseteq A^\uparrow$.
\end{itemize}

The {\em poset} $\mathbf P$ is called {\em distributive} (see e.g.\ \cite P) if it satisfies one of the following equivalent conditions:
\begin{align*}
 L\big(U(x,y),z\big) & \approx LU\big(L(x,z),L(y,z)\big), \\	
UL\big(U(x,y),z\big) & \approx U\big(L(x,z),L(y,z)\big), \\	
 U\big(L(x,y),z\big) & \approx UL\big(U(x,z),U(y,z)\big), \\
LU\big(L(x,y),z\big) & \approx L\big(U(x,z),U(y,z)\big).
\end{align*}
The {\em poset} $\mathbf P$ is called {\em bounded} if it has a least element $0$ and a greatest element $1$. Such a poset will be denoted by $(P,\le,0,1)$ and in the whole paper we assume that $0$ differs from $1$, i.e.\ the {\em bounded poset} is {\em non-trivial}.

A reflexive and transitive relation on a set is called a {\em quasiorder}, in particular, the relations $\le_1$ and $\le_2$ are quasiorders on the power set $2^P$ of $P$.

If the supremum or infimum of $A$ exists in the poset $\mathbf P$, we denote it by $\bigvee A$ and $\bigwedge A$, respectively.

Now let $\mathbf P=(P,\le,0,1)$ be a bounded poset. Recall that $b$ is called a {\em complement} of $a$ if $a\vee b=1$ and $a\wedge b=0$. We write $a\perp b$ in this case. This relation is considered as {\em orthogonality}. The {\em poset} $\mathbf P$ is called {\em complemented} if every element of $P$ has at least one complement. A distributive complemented poset is called a {\em Boolean poset}. A unary operation $\,'$ on $P$ is called an {\em involution} if it satisfies the identity $(x')'\approx x$, it is called {\em antitone} if $x,y\in P$ and $x\le y$ imply $y'\le x'$ and it is called a {\em complementation} if $x\perp x'$ for all $x\in P$. A complementation that is an antitone involution is called an {\em orthocomplementation}. An {\em ortholattice} is a bounded lattice $(P,\vee,\wedge,{}',0,1)$ with an orthcomplementation.

The following result was already proved by the authors in \cite{CL25a}. For the reader's convenience we repeat the proof.

\begin{proposition}\label{prop2}
Let $(P,\le,0,1)$ be a bounded distributive poset and $a,b,c\in P$ with $a\perp b$ and $a\perp c$. Then $b=c$ and hence the complementation in a Boolean poset is a unary operation that is an antitone involution {\rm(}cf.\ {\rm\cite N)}.
\end{proposition}

\begin{proof}
We have
\begin{align*}
L(b) & =L(1,b)=L\big(U(c,a),b\big)=LU\big(L(c,b),L(a,b)\big)=LU\big(L(c,b),0\big)=LU\big(L(b,c),0\big)= \\
     & =LU\big(L(b,c),L(a,c)\big)=L\big(U(b,a),c\big)=L(1,c)=L(c).	
\end{align*}	
\end{proof}

It was shown already by R.~P.~Dilworth \cite D that a finite uniquely complemented lattice is Boolean. By the previous Proposition~\ref{prop2}, every complemented distributive poset has a unique complementation and hence is Boolean. However, contrary to the case of lattices, there exist finite uniquely complemented posets which are not Boolean, see the following example.

\begin{example}\label{ex1}
The poset depicted in Fig.~1

\vspace*{-4mm}

\begin{center}
\setlength{\unitlength}{7mm}
\begin{picture}(14,10)
\put(7,1){\circle*{.3}}
\put(1,3){\circle*{.3}}
\put(4,3){\circle*{.3}}
\put(7,3){\circle*{.3}}
\put(10,3){\circle*{.3}}
\put(13,3){\circle*{.3}}
\put(2,5){\circle*{.3}}
\put(12,5){\circle*{.3}}
\put(1,7){\circle*{.3}}
\put(4,7){\circle*{.3}}
\put(7,7){\circle*{.3}}
\put(10,7){\circle*{.3}}
\put(13,7){\circle*{.3}}
\put(7,9){\circle*{.3}}
\put(7,1){\line(-3,1)6}
\put(7,1){\line(-3,2)3}
\put(7,1){\line(0,1)2}
\put(7,1){\line(3,2)3}
\put(7,1){\line(3,1)6}
\put(1,3){\line(1,2)1}
\put(1,3){\line(3,2)6}
\put(1,3){\line(9,4)9}
\put(4,3){\line(-1,1)2}
\put(4,3){\line(3,4)3}
\put(4,3){\line(9,4)9}
\put(7,3){\line(-3,2)6}
\put(7,3){\line(-3,4)3}
\put(7,3){\line(3,4)3}
\put(7,3){\line(3,2)6}
\put(10,3){\line(-9,4)9}
\put(10,3){\line(3,4)3}
\put(10,3){\line(1,1)2}
\put(13,3){\line(-9,4)9}
\put(13,3){\line(-1,2)1}
\put(13,3){\line(0,1)4}
\put(2,5){\line(-1,2)1}
\put(2,5){\line(1,1)2}
\put(12,5){\line(-5,2)5}
\put(12,5){\line(-1,1)2}
\put(7,9){\line(-3,-1)6}
\put(7,9){\line(-3,-2)3}
\put(7,9){\line(0,-1)2}
\put(7,9){\line(3,-2)3}
\put(7,9){\line(3,-1)6}
\put(6.85,.3){$0$}
\put(.35,2.85){$a$}
\put(3.35,2.85){$b$}
\put(7.4,2.85){$c$}
\put(10.4,2.85){$d$}
\put(13.4,2.85){$e$}
\put(1.35,4.85){$f$}
\put(12.4,4.85){$f'$}
\put(.35,6.85){$e'$}
\put(3.35,6.85){$d'$}
\put(7.4,6.85){$c'$}
\put(10.4,6.85){$b'$}
\put(13.4,6.85){$a'$}
\put(6.85,9.4){$1$}
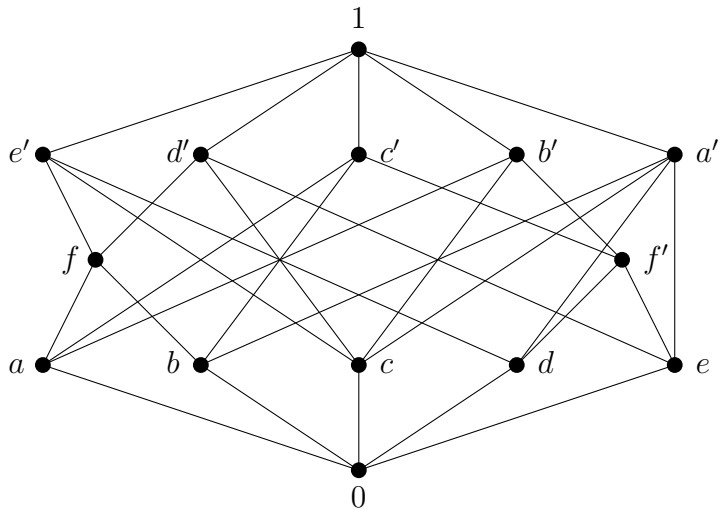
\put(1.7,-.75){{\rm Figure~1. Uniquely complemented poset}}
\end{picture}
\end{center}

\vspace*{4mm}

is evidently uniquely complemented but not Boolean since it is not distributive, e.g.
\[
L\big(U(c,f'),a\big)=L(b',1,a)=\{0,a\}\ne\{0\}=LU(0)=LU\big(L(c,a),L(f',a)\big).
\]
\end{example}

It was proved by J.~Niederle \cite N that every pseudocomplemented and uniquely complemented poset is distributive and hence Boolean. Example~\ref{ex1} is not in conflict with this result because the poset depicted in Fig.~1 is not pseudocomplemented. Namely, we have $L(a,a')=L(a,f')=0$ but there does not exist some $x\in P$ with $a',f'\le x$ and $L(a,x)=0$.

\section{The operator $^+$}

Let $\mathbf P=(P,\le,0,1)$ be a bounded poset. We define the following unary operator on subsets of $P$:
\[
A^+:=\{x\in P\mid x\perp y\text{ for all }y\in A\}
\]
for $A\subseteq P$. Hence, if $(P,\le,{}',0,1)$ is a Boolean poset then $x^+=x'$ for all $x\in P$. Obviously, $(^+,^+)$ is the Galois correspondence between the posets $(2^A,\subseteq)$ and $(2^A,\subseteq)$ induced by the symmetric relation $\perp$. From this fact we obtain immediately

\begin{lemma}
Let $\mathbf P=(P,\le,0,1)$ be a complemented poset and $A,B\subseteq P$. Then
\begin{enumerate}[{\rm(i)}]
\item $A\subseteq (A^+)^+$,
\item $A\subseteq B$ implies $B^+\subseteq A^+$,
\item $\big((A^+)^+\big)^+=A^+$,
\item $A\subseteq B^+$ is equivalent to $B\subseteq A^+$.
\end{enumerate}
\end{lemma}

Since $\mathbf P$ is non-trivial, we have $\emptyset^+=P$ and $P^+=\emptyset$.

A {\em subset} $A$ of $P$ will be called {\em closed} if $(A^+)^+=A$. Let $\Cl(\mathbf P)$ denote the set of all closed subsets of $P$. Then clearly $\Cl(\mathbf P)=\{A^+\mid A\subseteq P\}$. Because of $A^+\cap A^{++}=\emptyset$ for all $A\subseteq P$ we have that $\BCl(\mathbf P):=\big(\Cl(\mathbf P),\subseteq,{}^+,\emptyset,P\big)$ forms a complete ortholattice with orthocomplementation $^+$ and
\begin{align*}
  \bigvee_{i\in I}A_i & =\left(\left(\bigcup_{i\in I}A_i\right)^+\right)^+, \\
\bigwedge_{i\in I}A_i & =\bigcap_{i\in I}A_i
\end{align*}
for all families $(A_i;i\in I)$ of closed subsets of $P$.

Observe that $A^+=\bigcap\limits_{x\in A}x^+$ for all $A\subseteq P$ and that in case that $\mathbf P$ is uniquely complemented, in particular if $\mathbf P$ is Boolean, we have $\Cl(\mathbf P)=\{\emptyset,P\}\cup\{x^+\mid x\in P\}$.

In the following let $\mathbf N_5=(N_5,\vee,\wedge)$ denote the lattice visualized in Fig.~2.
	
\vspace*{-4mm}
	
\begin{center}
\setlength{\unitlength}{7mm}
\begin{picture}(4,8)
\put(2,1){\circle*{.3}}
\put(1,4){\circle*{.3}}
\put(3,3){\circle*{.3}}
\put(3,5){\circle*{.3}}
\put(2,7){\circle*{.3}}
\put(2,1){\line(-1,3)1}
\put(2,1){\line(1,2)1}
\put(3,3){\line(0,1)2}
\put(2,7){\line(-1,-3)1}
\put(2,7){\line(1,-2)1}
\put(1.85,.3){$0$}
\put(3.4,2.85){$a$}
\put(3.4,4.85){$c$}
\put(.35,3.85){$b$}
\put(1.85,7.4){$1$}
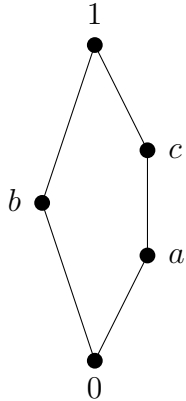
\put(-2.2,-.75){{\rm Figure~2. Non-modular lattice $\mathbf N_5$}}
\end{picture}
\end{center}
	
\vspace*{10mm}

Of course, this lattice is complemented, e.g.\ we have $b^+=\{a,c\}$. The lattice $\BCl(\mathbf N_5)$ is depicted in Fig.~3.

\vspace*{-4mm}
	
\begin{center}
\setlength{\unitlength}{7mm}
\begin{picture}(8,6)
\put(4,1){\circle*{.3}}
\put(1,3){\circle*{.3}}
\put(3,3){\circle*{.3}}
\put(5,3){\circle*{.3}}
\put(7,3){\circle*{.3}}
\put(4,5){\circle*{.3}}
\put(4,1){\line(-3,2)3}
\put(4,1){\line(-1,2)1}
\put(4,1){\line(1,2)1}
\put(4,1){\line(3,2)3}
\put(4,5){\line(-3,-2)3}
\put(4,5){\line(-1,-2)1}
\put(4,5){\line(1,-2)1}
\put(4,5){\line(3,-2)3}
\put(3.85,.3){$\emptyset$}
\put(.35,2.85){$0$}
\put(2.35,2.85){$b$}
\put(5.4,2.85){$ac$}
\put(7.4,2.85){$1$}
\put(3.7,5.4){$N_5$}
\put(.2,-.75){{\rm Figure~3. The lattice $\BCl(\mathbf N_5)$}}
\end{picture}
\end{center}
	
\vspace*{4mm}
	
Here we shortly write $ac$ for $\{a,c\}$. Similarly we proceed in other figures.
	
Some basic properties of the operator $^+$ are as follows:

\begin{proposition}\label{prop1}
Let $\mathbf P=(P,\le,0,1)$ be a complemented poset and $a\in P$. Then the following hold:
\begin{enumerate}[{\rm(i)}]
\item $a\in(a^+)^+$ and $\big((a^+)^+\big)^+=a^+$,
\item $(x^+,\le)$ is an antichain for every $x\in P$ if and only if $\mathbf P$ does not contain a sublattice isomorphic to $\mathbf N_5$ containing $0$ and $1$,
\item $(a^+,\le)$ is convex,
\item if the mapping $x\mapsto(x^+)^+$ from $L$ to $2^L$ is not injective then $\mathbf L$ does not satisfy the identity $(x^+)^+\approx x$.
\end{enumerate}
\end{proposition}

\begin{proof}
\
\begin{enumerate}[(i)]
\item follows directly from above.
\item First assume there exists some $b\in P$ such that $(b^+,\le)$ is not an antichain. Then $b\notin\{0,1\}$. Now there exist $c,d\in b^+$ with $c<d$. Since $b\notin\{0,1\}$ and $b\in c^+\cap d^+$ we have $c,d\notin\{0,1\}$. Because of $|P|>1$ we have $b\notin\{c,d\}$. Hence the elements $0$, $b$, $c$, $d$ and $1$ are pairwise distinct and form an $\mathbf N_5$ containing $0$ and $1$. If, conversely, $\mathbf P$ contains a sublattice $(P_1,\vee,\wedge)$ isomorphic to $\mathbf N_5$ and containing $0$ and $1$, say $P_1=\{0,e,f,g,1\}$ with $e<f$ then $e,f\in g^+$ and hence $(g^+,\le)$ is not an antichain.
\item If $b,c\in a^+$, $d\in P$ and $b\le d\le c$ then $\{1\}\subseteq U(a,d)\subseteq U(a,b)=\{1\}$ and $\{0\}\subseteq L(a,d)\subseteq L(a,c)=\{0\}$ showing $d\in a^+$.
\item If the mapping $x\mapsto (x^+)^+$ is not injective then there exist $a,b\in L$ with $a\ne b$ and $(a^+)^+=(b^+)^+$ which implies $b\in(b^+)^+=(a^+)^+$ and $b\ne a$ and hence $(a^+)^+\ne a$ showing that $\mathbf P$ does not satisfy the identity $(x^+)^+\approx x$.
\end{enumerate}
\end{proof}

Next we present an example of a complemented poset $(P,\le,0,1)$ where $(x^+,\le)$ is an antichain for each $x\in P$.

\begin{example}
The complemented poset $\mathbf P=(P,\le,0,1)$ visualized in Fig.~4

\vspace*{-4mm}

\begin{center}
\setlength{\unitlength}{7mm}
\begin{picture}(8,8)
\put(4,1){\circle*{.3}}
\put(1,3){\circle*{.3}}
\put(3,3){\circle*{.3}}
\put(5,3){\circle*{.3}}
\put(7,3){\circle*{.3}}
\put(1,5){\circle*{.3}}
\put(3,5){\circle*{.3}}
\put(5,5){\circle*{.3}}
\put(7,5){\circle*{.3}}
\put(4,7){\circle*{.3}}
\put(4,1){\line(-3,2)3}
\put(4,1){\line(-1,2)1}
\put(4,1){\line(1,2)1}
\put(4,1){\line(3,2)3}
\put(4,7){\line(-3,-2)3}
\put(4,7){\line(-1,-2)1}
\put(4,7){\line(1,-2)1}
\put(4,7){\line(3,-2)3}
\put(1,3){\line(0,1)2}
\put(1,3){\line(1,1)2}
\put(1,3){\line(2,1)4}
\put(3,3){\line(-1,1)2}
\put(3,3){\line(2,1)4}
\put(5,3){\line(-2,1)4}
\put(5,3){\line(1,1)2}
\put(7,3){\line(-2,1)4}
\put(7,3){\line(-1,1)2}
\put(7,3){\line(0,1)2}
\put(3.85,.3){$0$}
\put(.35,2.85){$a$}
\put(2.35,2.85){$b$}
\put(5.4,2.85){$c$}
\put(7.4,2.85){$d$}
\put(.35,4.85){$e$}
\put(2.35,4.85){$f$}
\put(5.4,4.85){$g$}
\put(7.4,4.85){$h$}
\put(3.85,7.4){$1$}
\put(-2.2,-.75){{\rm Figure~4. Non-distributive complemented poset $\mathbf P$}}
\end{picture}
\end{center}

\vspace*{4mm}

is not distributive since
\[
L\big(U(a,b),c\big)=L(e,1,c)=L(c)\ne0=LU(0,0)=LU\big(L(a,c),L(b,c)\big)
\]
and does not contain a sublattice isomorphic to $\mathbf N_5$ containing $0$ and $1$. Accordingly to Proposition~\ref{prop1} {\rm(ii)}, $(b^+,\le)=(c^+,\le)=(\{f,g\},\le)$ is an antichain. The lattice $\BCl(\mathbf P)$ is depicted in Fig.~5.

\vspace*{-4mm}

\begin{center}
\setlength{\unitlength}{7mm}
\begin{picture}(16,6)
\put(8,1){\circle*{.3}}
\put(1,3){\circle*{.3}}
\put(3,3){\circle*{.3}}
\put(5,3){\circle*{.3}}
\put(7,3){\circle*{.3}}
\put(9,3){\circle*{.3}}
\put(11,3){\circle*{.3}}
\put(13,3){\circle*{.3}}
\put(15,3){\circle*{.3}}
\put(8,5){\circle*{.3}}
\put(8,1){\line(-7,2)7}
\put(8,1){\line(-5,2)5}
\put(8,1){\line(-3,2)3}
\put(8,1){\line(-1,2)1}
\put(8,1){\line(7,2)7}
\put(8,1){\line(5,2)5}
\put(8,1){\line(3,2)3}
\put(8,1){\line(1,2)1}
\put(8,5){\line(-3,-2)3}
\put(8,5){\line(-1,-2)1}
\put(8,5){\line(1,-2)1}
\put(8,5){\line(3,-2)3}
\put(8,5){\line(-5,-2)5}
\put(8,5){\line(-7,-2)7}
\put(8,5){\line(5,-2)5}
\put(8,5){\line(7,-2)7}
\put(7.85,.3){$\emptyset$}
\put(.35,2.85){$0$}
\put(2.35,2.85){$a$}
\put(4.35,2.85){$d$}
\put(6.1,2.85){$bc$}
\put(9.4,2.85){$fg$}
\put(11.4,2.85){$e$}
\put(13.4,2.85){$h$}
\put(15.4,2.85){$1$}
\put(7.7,5.4){$P$}
\put(4.4,-.75){{\rm Figure~5. The lattice $\BCl(\mathbf P)$}}
\end{picture}
\end{center}

\vspace*{4mm}

\end{example}

An example of a complemented poset which is not a lattice where $(x^+,\le)$ need not be an antichain is as follows.

\begin{example}
The complemented poset $\mathbf P=(P,\le,0,1)$ visualized in Fig.~6

\vspace*{-4mm}

\begin{center}
\setlength{\unitlength}{7mm}
\begin{picture}(10,8)
\put(5,1){\circle*{.3}}
\put(1,3){\circle*{.3}}
\put(3,3){\circle*{.3}}
\put(5,3){\circle*{.3}}
\put(7,3){\circle*{.3}}
\put(1,5){\circle*{.3}}
\put(3,5){\circle*{.3}}
\put(5,5){\circle*{.3}}
\put(7,5){\circle*{.3}}
\put(5,7){\circle*{.3}}
\put(9,3){\circle*{.3}}
\put(9,5){\circle*{.3}}
\put(5,1){\line(-2,1)4}
\put(5,1){\line(-1,1)2}
\put(5,1){\line(0,1)2}
\put(5,1){\line(1,1)2}
\put(5,1){\line(2,1)4}
\put(5,7){\line(-2,-1)4}
\put(5,7){\line(-1,-1)2}
\put(5,7){\line(0,-1)2}
\put(5,7){\line(1,-1)2}
\put(5,7){\line(2,-1)4}
\put(9,3){\line(0,1)2}
\put(1,3){\line(0,1)2}
\put(1,3){\line(1,1)2}
\put(1,3){\line(2,1)4}
\put(3,3){\line(-1,1)2}
\put(3,3){\line(2,1)4}
\put(5,3){\line(-2,1)4}
\put(5,3){\line(1,1)2}
\put(7,3){\line(-2,1)4}
\put(7,3){\line(-1,1)2}
\put(7,3){\line(0,1)2}
\put(4.85,.3){$0$}
\put(.35,2.85){$a$}
\put(2.35,2.85){$b$}
\put(5.4,2.85){$c$}
\put(7.4,2.85){$d$}
\put(9.4,2.85){$i$}
\put(.35,4.85){$e$}
\put(2.35,4.85){$f$}
\put(5.4,4.85){$g$}
\put(7.4,4.85){$h$}
\put(9.4,4.85){$j$}
\put(4.85,7.4){$1$}
\put(-1,-.75){{\rm Figure~6. Non-distributive complemented poset}}
\end{picture}
\end{center}

\vspace*{4mm}

is the horizontal sum of the poset from Fig.~4 and the four-element chain $0<i<j<1$ and hence it is not distributive. Moreover, $\mathbf P$ contains the sublattice $\{0,a,e,i\}$ isomorphic to $\mathbf N_5$ and containing $0$ and $1$, and hence $(b^+,\le)=(\{f,g,i,j\},\le)$ is not an antichain in accordance with Proposition~\ref{prop1} {\rm(ii)}.
\end{example}

\begin{proposition}
Let $(P,\le,0,1)$ be a finite complemented poset such that $x\mapsto(x^+)^+$ is injective and $a\in L$ and assume $(a^+)^+\ne a$. Then there exists some $b\in(a^+)^+$ with $(b^+)^+=b$.
\end{proposition}

\begin{proof}
Let $a_1\in(a^+)^+\setminus\{a\}$. Then $(a_1^+)^+\subseteq(a^+)^+$. Since $a_1\ne a$ and $x\mapsto(x^+)^+$ is injective we conclude $(a_1^+)^+\subsetneqq(a^+)^+$. Now either $(a_1^+)^+=a_1$ or there exists some $a_2\in(a_1^+)^+\setminus\{a_1\}$. Then $(a_2^+)^+\subseteq(a_1^+)^+$. Since $a_2\ne a_1$ and $x\mapsto(x^+)^+$ is injective we conclude $(a_2^+)^+\subsetneqq(a_1^+)^+$. Now either $(a_2^+)^+=a_2$ or there exists some $a_3\in(a_2^+)^+\setminus\{a_2\}$. Since $P$ is finite and $(a_1^+)^+\supsetneqq (a_2^+)^+\supsetneqq\cdots$ there exists some $n\ge1$ with $|(a_n^+)^+|=1$, i.e.\ $(a_n^+)^+=a_n$ and we have $a_n\in(a_n^+)^+\subseteq(a_{n-1}^+)^+\subseteq\cdots\subseteq( a_1^+)^+\subseteq(a^+)^+$.
\end{proof}

A {\em poset} is said to be {\em of finite length} if it has no infinite chains. Let $(P,\le)$ be a poset of finite length and $A\subseteq P$. Then $\Max A\ne\emptyset$, $\Min A\ne\emptyset$, $U(A)=U(\Max A)$, $L(A)=L(\Min A)$, $A\le_1\Max A$ and $\Min A\le_2A$.

\begin{lemma}\label{lem1}
Let $\mathbf P=(P,\le)$ be a poset of fine length and $A,B\subseteq P$ with $A\subseteq B$. Then the following holds:
\begin{enumerate}[{\rm(i)}]
\item $\Max A\le_1\Max B$ and $\Max L(B)\le_1\Max L(A)$,
\item $\Min B\le_2\Min A$ and $\Min U(A)\le_2\Min U(B)$.
\end{enumerate}
\end{lemma}

\begin{proof}
According to the remark at the beginning of this section we have $A\le_1\Max A$. Now $\Max A\subseteq A\subseteq B\le_1\Max B$ and therefore $\Max A\le_1\Max B$. Finally we conclude $L(B)\subseteq L(A)$ and therefore $\Max L(B)\le_1\Max L(A)$. This shows (i). The proof of (ii) is dual to that of (i).
\end{proof}

In what follows we study connections among the sets $\Min U(x^+,y^+)$, $\Max L(x^+,y^+)$, $\big(\Min U(x,y)\big)^+$ and $\big(\Max L(x,y)\big)^+$.

\begin{theorem}\label{th2}
Let $\mathbf P=(P,\le,0,1)$ be a complemented poset of finite length and consider the following statements:
\begin{enumerate}[{\rm(i)}]
\item for all $x,y\in P$, $x\le y$ implies $y^+\le x^+$,
\item for all $x,y\in P$, $x\le y$ implies $y^+\le_1x^+$,
\item for all $x,y\in P$, $x\le y$ implies $y^+\le_2x^+$,
\item $\big(\Min U(x,y)\big)^+\le_1\Max L(x^+,y^+)$ for all $x,y\in P$,
\item $\Min U(x^+,y^+)\le_2\big(\Max L(x,y)\big)^+$ for all $x,y\in P$.
\end{enumerate}
Then {\rm(i)} implies {\rm(ii)} and {\rm(iii)}. Moreover, {\rm(iv)} implies {\rm(ii)}, and {\rm(v)} implies {\rm(iii)}. If for all $x\in P$, $x^+$ has a greatest element $x_2$ then {\rm(i)} implies {\rm(iv)}. If, finally, for all $x\in P$, $x^+$ has a smallest element $x_1$ then {\rm(i)} implies {\rm(v)}.
\end{theorem}

\begin{proof}
Let $a,b\in P$. \\
That (i) implies (ii) and (iii) is clear. \\
If $\mathbf P$ satisfies (iv) and $a\le b$ then
\[
b^+=\big(\Min U(a,b)\big)^+\le_1\Max L(a^+,b^+)\le a^+.
\]
If $\mathbf P$ satisfies (v) and $a\le b$ then
\[
b^+=\Min U(a^+,b^+)\le_2\big(\Max L(a,b)\big)^+=a^+.
\]
Now assume (i) and that for all $x\in P$, $x^+$ has a greatest element $x_2$. Let $c\in\big(\Min U(a,b)\big)^+$ and $d\in\Min U(a,b)$. Then $a,b\le d$ and hence $d^+\le a^+,b^+$ which implies $c\le d_2\le a^+,b^+$. Thus $c\in L(a^+,b^+)$ showing $\big(\Min U(a,b)\big)^+\subseteq L(a^+,b^+)$ whence $\big(\Min U(a,b)\big)^+\le_1\Max L(a^+,b^+)$. \\
Finally, assume (i) and that for all $x\in P$, $x^+$ has a smallest element $x_1$. Let $e\in\big(\Max L(a,b)\big)^+$ and $f\in\Max L(a,b)$. Then $f\le a,b$ and hence $a^+,b^+\le f^+$ which implies $a^+,b^+\le f_1\le e$. Thus $e\in U(a^+,b^+)$ showing $\big(\Max L(a,b)\big)^+\subseteq U(a^+,b^+)$ whence $\Min U(a^+,b^+)\le_2\big(\Max L(a,b)\big)^+$.
\end{proof}

The next example shows that the conditions (i) -- (iii) of Theorem~\ref{th2} are far from being trivial, namely, we provide a complemented poset which is not a lattice where these conditions are violated.

\begin{example}
Consider the complemented poset depicted in Fig.~7

\vspace*{-4mm}

\begin{center}
\setlength{\unitlength}{7mm}
\begin{picture}(18,20)
\put(9,1){\circle*{.3}}
\put(4,6){\circle*{.3}}
\put(3,7){\circle*{.3}}
\put(15,7){\circle*{.3}}
\put(1,9){\circle*{.3}}
\put(5,9){\circle*{.3}}
\put(9,9){\circle*{.3}}
\put(13,9){\circle*{.3}}
\put(17,9){\circle*{.3}}
\put(1,11){\circle*{.3}}
\put(5,11){\circle*{.3}}
\put(9,11){\circle*{.3}}
\put(13,11){\circle*{.3}}
\put(17,11){\circle*{.3}}
\put(3,13){\circle*{.3}}
\put(15,13){\circle*{.3}}
\put(14,14){\circle*{.3}}
\put(9,19){\circle*{.3}}
\put(9,1){\line(-1,1)8}
\put(9,1){\line(0,1)9}
\put(9,10){\line(0,1)9}
\put(9,1){\line(1,1)8}
\put(9,19){\line(-1,-1)8}
\put(9,19){\line(1,-1)8}
\put(9,9){\line(1,1)5}
\put(9,11){\line(-1,-1)5}
\put(3,7){\line(1,1)2}
\put(1,9){\line(2,1)4}
\put(1,11){\line(2,-1)4}
\put(3,13){\line(1,-1)2}
\put(15,7){\line(-1,1)2}
\put(17,9){\line(-2,1)4}
\put(17,11){\line(-2,-1)4}
\put(15,13){\line(-1,-1)2}
\put(1,9){\line(0,1)2}
\put(5,9){\line(0,1)2}
\put(13,9){\line(0,1)2}
\put(17,9){\line(0,1)2}
\put(8.85,.3){$0$}
\put(8.85,19.4){$1$}
\put(3.35,5.85){$a$}
\put(2.35,6.85){$b$}
\put(.35,8.85){$d$}
\put(15.4,6.85){$c$}
\put(17.4,8.85){$h$}
\put(.35,10.85){$i$}
\put(2.35,12.85){$n$}
\put(17.4,10.85){$m$}
\put(15.4,12.85){$o$}
\put(14.4,13.85){$p$}
\put(5.4,8.85){$e$}
\put(8.35,8.85){$f$}
\put(12.35,8.85){$g$}
\put(5.4,10.85){$j$}
\put(12.35,10.85){$l$}
\put(9.4,10.85){$k$}
\put(4.8,-.75){{\rm Figure~7. Complemented poset}}
\end{picture}
\end{center}

\vspace*{4mm}

We have $f\le k$, $f^+=[b,n]$ and $k^+=[c,o]$ and hence $k^+\not\le f^+$, $k^+\not\le_1f^+$ and $k^+\not\le_2f^+$.
\end{example}

That the generalized De Morgan's laws do not hold in general is shown by the following lemma.

\begin{lemma}
Let $\mathbf P=(P,\le,0,1)$ be a complemented poset of finite length that is not uniquely complemented. Then the {\em generalized de Morgan's laws}
\begin{align*}
\big(\Min U(x,y)\big)^+ & \approx\Max L(x^+,y^+), \\	
\big(\Max L(x,y)\big)^+ & \approx\Min U(x^+,y^+)
\end{align*}
do not hold.
\end{lemma}

\begin{proof}
Since $\mathbf P$ is complemented, but not uniquely complemented, there exists some $a\in P$ with $|a^+|>1$. Let $b,c\in a^+$ with $b\ne c$. Suppose $\big(\Min U(x,y)\big)^+\approx\Max L(x^+,y^+)$ holds. Then
\[
a^+=\big(\Min U(a,0)\big)^+=\Max L(a^+,0^+)=\Max L(a^+)\subseteq L(a^+).
\]
Since $b,c\in a^+$ we have $b\in a^+\subseteq L(a^+)\subseteq L(c)$, i.e. $b\le c$. From symmetry reasons we have also $c\le b$ and hence $b=c$, a contradiction. Now suppose $\big(\Max L(x,y)\big)^+\approx\Min U(x^+,y^+)$ holds. Then
\[
a^+=\big(\Max L(a,1)\big)^+=\Min U(a^+,1^+)=\Min U(a^+)\subseteq U(a^+).
\]
Since $b,c\in a^+$ we have $c\in a^+\subseteq U(a^+)\subseteq U(b)$, i.e. $b\le c$. From symmetry reasons we have also $c\le b$ and hence $b=c$, a contradiction.
\end{proof}

\begin{example}
The complemented poset $(P,\le,0,1)$ visualized in Fig.~8
	
\vspace*{-4mm}
	
\begin{center}
\setlength{\unitlength}{7mm}
\begin{picture}(8,8)
\put(4,1){\circle*{.3}}
\put(1,3){\circle*{.3}}
\put(3,3){\circle*{.3}}
\put(5,3){\circle*{.3}}
\put(7,3){\circle*{.3}}
\put(1,5){\circle*{.3}}
\put(3,5){\circle*{.3}}
\put(5,5){\circle*{.3}}
\put(7,5){\circle*{.3}}
\put(4,7){\circle*{.3}}
\put(4,1){\line(-3,2)3}
\put(4,1){\line(-1,2)1}
\put(4,1){\line(1,2)1}
\put(4,1){\line(3,2)3}
\put(4,7){\line(-3,-2)3}
\put(4,7){\line(-1,-2)1}
\put(4,7){\line(1,-2)1}
\put(4,7){\line(3,-2)3}
\put(1,3){\line(0,1)2}
\put(1,3){\line(1,1)2}
\put(3,3){\line(-1,1)2}
\put(3,3){\line(0,1)2}
\put(5,3){\line(0,1)2}
\put(5,3){\line(1,1)2}
\put(7,3){\line(-1,1)2}
\put(7,3){\line(0,1)2}
\put(3.85,.3){$0$}
\put(.35,2.85){$a$}
\put(3.4,2.85){$b$}
\put(4.35,2.85){$c$}
\put(7.4,2.85){$d$}
\put(.35,4.85){$e$}
\put(3.4,4.85){$f$}
\put(4.35,4.85){$g$}
\put(7.4,4.85){$h$}
\put(3.85,7.4){$1$}
\put(-2,-.75){{\rm Figure~8. Non-distributive complemented poset}}
\end{picture}
\end{center}
	
\vspace*{4mm}

does not satisfy the generalized De Morgan's laws since
\begin{align*}
\big(\Min U(a,0)\big)^+ & =a^+=\{c,d,g,h\}\ne\{0\}=\Max L(a^+)=\Max L(a^+,0^+), \\
\big(\Max L(a,1)\big)^+ & =a^+=\{c,d,g,h\}\ne\{1\}=\Min U(a^+)=\Min U(a^+,1^+).
\end{align*}
\end{example}

\section{Derived operators}

Using the operator $^+$ we define other important operators in complemented posets of finite length as follows:

\begin{definition}
Let $(P,\le,0,1)$ be a complemented poset of finite length and $a,b\in P$. Define
\begin{align*}
          a\circ b & :=\Max L(a,b), \\
            a\to b & :=\Min U(a^+,b), \\
          a\odot b & :=\Max L\big(b,U(a,b^+)\big), \\
a\hookrightarrow b & :=\Min U\big(a^+,L(a,b)\big).
\end{align*}
\end{definition}

Observe that e.g.
\begin{align*}
\Min U\big(a^+,L(a,b)\big) & =\Min\big(U(a^+)\cap UL(a,b)\big)=\Min\Big(U(a^+)\cap U\big(\Max L(a,b)\big)\Big)= \\
                           & =\Min U\big(a^+,\Max L(a,b)\big).
\end{align*}

The following can be easily seen:
\begin{align*}
   x\circ y=y\circ x, & \text{ and }x\circ y=x\text{ if }x\le y, \\
            1\to x=x, & \text{ and }x\to y=1\text{ if }x\le y, \\
           x\odot1=x, & \text{ and }y\odot x=x\text{ if }x\le y, \\
1\hookrightarrow x=x, & \text{ and }x\hookrightarrow y=1\text{ if }x\le y.
\end{align*}

The couple $(\circ,\to)$ is said to form an {\em adjoint pair} if
\[
x\circ y\le z\text{ is equivalent to }x\le y\to z
\]
for all $x,y,z\in P$.

\begin{definition}
We call a binary operator $*$ on a complemented poset $(P,\le,0,1)$ of finite length
\begin{itemize}
\item {\em monotone from the left} if $x\le y$ implies $x*z\le_1y*z$,
\item {\em monotone from the right} if $x\le y$ implies $z*x\le_2z*y$,
\item {\em weakly antitone from the left} if $x*y\ge1*y$,
\item {\em weakly monotone from the right} if $x*y\le x*1$
\end{itemize}
{\rm(}$x,y,z\in P${\rm)}.
\end{definition}

\begin{lemma}
Let $(P,\le,0,1)$ be a complemented poset of finite height. Then the following holds:
\begin{enumerate}[{\rm(i)}]
\item $\circ$ and $\odot$ are monotone from the left,
\item $\to$ and $\hookrightarrow$ are monotone from the right.
\end{enumerate}	
\end{lemma}

\begin{proof}
Let $a,b,c\in P$ with $a\le b$.
\begin{enumerate}[(i)]
\item We have $L(a,c)\subseteq L(b,c)$ and hence
\[
a\circ c=\Max L(a,c)\le_1\Max L(b,c)=b\circ c.
\]
Moreover, we have $\{c\}\cup U(b,c^+)\subseteq\{c\}\cup U(a,c^+)$ and hence
\[
L\big(c,U(a,c^+)\big)\subseteq L\big(c,U(b,c^+)\big)
\]
whence
\[
a\odot c=\Max L\big(c,U(a,c^+)\big)\le_1\Max L\big(c,U(b,c^+)\big)=b\odot c.
\]
\item We have $U(c^+,b)\subseteq U(c^+,a)$ and hence
\[
c\to a=\Min U(c^+,a)\le_2\Min U(c^+,b)=c\to b.
\]
Moreover, we have $\{c^+\}\cup L(c,a)\subseteq\{c^+\}\cup L(c,b)$ and hence
\[
U\big(c^+,L(c,b)\big)\subseteq U\big(c^+,L(c,a)\big)
\]
whence
\[
c\hookrightarrow a=\Min U\big(c^+,L(c,a)\big)\le_2\Min U\big(c^+,L(c,b)\big)=c\hookrightarrow b.
\]
\end{enumerate}	
\end{proof}

On the other hand, weak monotonicity from the right of the operator $\odot$ yields that the complementation of the poset in question is unique.

\begin{lemma}
Let $\mathbf P=(P,\le,0,1)$ be a complemented poset of finite length and assume $\odot$ to be weakly monotone from the right. Then $\mathbf P$ is uniquely complemented.
\end{lemma}

\begin{proof}
Let $a\in P$ and $b,c\in a^+$. Then $U(c,b^+)=1$ since $a\in b^+$ and $c\vee a=1$. Hence
\[
b=\Max L(b,1)=\Max L\big(b,U(c,b^+)\big)=c\odot b\le c\odot1=c.
\]
Now $c\le b$ follows from symmetry reasons. Together we obtain $b=c$.	
\end{proof}

The next two results show how the properties of the pair $(\odot,\hookrightarrow)$ influence the behavior of the pair $(\circ,\to)$.

\begin{theorem}\label{th1}
Let $(P,\le,0,1)$ be a complemented poset of finite length and $a,b,c\in P$. Then the following hold:
\begin{enumerate}[{\rm(i)}]
\item If $\hookrightarrow$ is weakly antitone from the left then $a\circ b\le c$ implies $a\le b\to c$.
\item If $\odot$ is weakly monotone from the right then $a\le b\to c$ implies $a\circ b\le c$.
\end{enumerate}
\end{theorem}

\begin{proof}
The following are equivalent: $a\circ b\le c$, $\Max L(a,b)\le c$, $L(a,b)\le c$. Moreover, the following are equivalent: $a\le b\to c$, $a\le\Min U(b^+,c)$, $a\le U(b^+,c)$.
\begin{enumerate}[(i)]
\item If $\hookrightarrow$ is weakly antitone from the left and $a\circ b\le c$ then
\[
a=1\hookrightarrow a\le b\hookrightarrow a=\Min U\big(b^+,L(b,a)\big)
\]
and $L(a,b)\le c$ and hence
\[
U(b^+,c)\subseteq U\big(b^+,L(a,b)\big)\subseteq U(a),
\]
i.e.\ $a\le U(b^+,c)$ showing $a\le b\to c$.
\item If $\odot$ is weakly monotone from the right and $a\le b\to c$ then
\[
\Max L\big(b,U(c,b^+)\big)=c\odot b\le c\odot1=c
\]
and $a\le U(b^+,c)$ and hence
\[
L(a,b)\subseteq L\big(U(b^+,c),b\big)\subseteq L(c),	
\]
i.e.\ $L(a,b)\le c$ showing $a\circ b\le c$.
\end{enumerate}
\end{proof}

\begin{remark}
In a Boolean poset $(P,\le,{}',0,1)$ the operator $\hookrightarrow$ is weakly antitone from the left and the operator $\odot$ is weakly monotone from the right since
\begin{align*}
x\hookrightarrow y & =\Min U\big(x',L(x,y)\big)=\Min UL\big(U(x',x),U(x',y)\big)=\Min ULU(x',y)= \\
                   & =\Min U(x',y)\ge y=1\hookrightarrow y, \\
          x\odot y & =\Max L\big(y,U(x,y')\big)=\Max LU\big(L(y,x),L(y,y')\big)=\Max LUL(y,x)= \\
                   & =\Max L(y,x)\le x=x\odot1.	
\end{align*}
\end{remark}

\begin{corollary}
If $\mathbf P$ is a complemented poset of finite length, $\hookrightarrow$ is weakly antitone from the left and $\odot$ is weakly monotone from the right then $(\circ,\to)$ form an adjoint pair.
\end{corollary}

Of course, in a Boolean poset $(P,\le,{}',0,1)$ the operator $\hookrightarrow$ is weakly antitone from the left since in such a poset we have
\begin{align*}
x\hookrightarrow y & =\Min U\big(x',L(x,y)\big)=\Min U\big(L(x,y),x'\big)=\Min UL\big(U(x,x'),U(y,x')\big)= \\
                   & =\Min ULU(y,x')=\Min U(y,x')\ge y=1\hookrightarrow y.
\end{align*}

However, we can present an example of a non-uniquely complemented poset that is not a lattice and where $\hookrightarrow$ is weakly antitone from the left.

\begin{example}
The complemented poset $(P,\le,0,1)$ from Figure~8 is neither distributive since
\[
L\big(U(e,f),b\big)=L(1,b)=L(b)\ne0=LU(0,0)=LU\big(L(e,b),L(f,b)\big)
\]
nor a lattice since $a\vee b$ does not exist. Now let $x,y,z\in P$. Then
\begin{align*}
            U(x^+) & =\left\{
\begin{array}{ll}
P & \text{if }x=1, \\
1 & \text{otherwise},
\end{array}
\right. \\
           x\circ1 & =\Max L(x,1)=x, \\
            x\to y & =\Min U(x^+,y)=\left\{
\begin{array}{ll}
\Min U(0,y)=y & \text{if }x=1, \\
1             & \text{otherwise}.
\end{array}
\right. \\
x\hookrightarrow y & =\Min U\big(x^+,L(x,y)\big)=\left\{
\begin{array}{ll}
\Min U\big(0,L(1,y)\big)=y & \text{if }x=1, \\
1                          & \text{otherwise}.
\end{array}
\right.
\end{align*}
The assumption of {\rm(i)} of Theorem~\ref{th1} is satisfied since $x\hookrightarrow y\ge y=1\hookrightarrow y$. Accordingly, the conclusion of {\rm(i)} of Theorem~\ref{th1} holds: \\
If $x\circ y\le z$ and $y=1$ then $x=x\circ1=x\circ y\le z=1\to z=y\to z$. \\
If $x\circ y\le z$ and $y\ne1$ then $x\le1=y\to z$.
\end{example}

We now turn our attention again to the adjointness of the operators $\circ$ and $\to$. In the proof we use similar arguments as in \cite{CL19}.

\begin{theorem}\label{th3}
Let $(P,\le,0,1)$ be a complemented poset of finite length. If
\begin{equation}
U\big(x^+,L(x,y)\big)\subseteq U(y)\text{ for all }x,y\in P
\end{equation}
then for all $x,y,z\in P$
\[
x\circ y\le z\text{ implies }x\le y\to z.
\]
If
\begin{equation}
L\big(x,U(x^+,y)\big)\subseteq L(y)\text{ for all }x,y\in P
\end{equation}
then for all $x,y,z\in P$
\[
x\le y\to z\text{ implies }x\circ y\le z.
\]
\end{theorem}

\begin{proof}
Let $a,b,c\in P$. Then the following are equivalent:
\begin{align*}
   a\circ b & \le c, \\
\Max L(a,b) & \le c, \\
     L(a,b) & \le c, \\
     L(a,b) & \subseteq L(c), \\
       U(c) & \subseteq UL(a,b).
\end{align*}
Moreover, the following are equivalent;
\begin{align*}
a & \le b\to c, \\
a & \le\Min U(b^+,c), \\
a & \le U(b^+,c), \\
a & \in LU(b^+,c), \\
L(a) & \subseteq LU(b^+,c), \\
U(b^+,c) & \subseteq U(a).
\end{align*}
If $a\circ b\le c$ and (1) then $U(c)\subseteq UL(a,b)$ and hence
\[
U(b^+,c)=U(b^+)\cap U(c)\subseteq U(b^+)\cap UL(a,b)=LU\big(b^+,L(a,b)\big)\subseteq U(a)
\]
showing $a\le b\to c$. If, conversely, $a\le b\to c$ and (2) then $L(a)\subseteq LU(b^+,c)$ and hence
\[
L(a,b)=L(a)\cap L(b)\subseteq LU(b^+,c)\cap L(b)=L\big(U(b^+,c),b\big)\subseteq L(c)
\]
showing $a\odot b\le c$.
\end{proof}

\begin{example}
The poset from Figure~8 satisfies {\rm(1)} since
\[
U\big(x^+,L(x,y)\big)\left\{
\begin{array}{ll}
=U(y) 								 & \text{ if }x=1 \\
\subseteq U(x^+)=\{1\}\subseteq U(y) & \text{ otherwise}.
\end{array}
\right.
\]
\end{example}

\begin{remark}
Every Boolean poset $(P,\le,{}',0,1)$ satisfies {\rm(1)} and {\rm(2)} since
\begin{align*}
U\big(x',L(x,y)\big) & =UL\big(U(x',x),U(x',y)\big)=ULU(x',y)=U(x',y)\subseteq U(y), \\
L\big(x,U(x',y)\big) & =LU\big(L(x,x'),L(x,y)\big)=LUL(x,y)=L(x,y)\subseteq L(y).
\end{align*}
\end{remark}	

In the next theorem we prove a similar assertion as in Theorem~\ref{th3}, but for the operators $\odot$ and $\hookrightarrow$.

\begin{theorem}\label{th4}
Let $(P,\le,0,1)$ be a complemented poset of finite length. If
\begin{equation}
U\Big(x^+,L\big(x,U(x^+,y)\big)\Big)\subseteq U(y)\text{ for all }x,y\in P
\end{equation}
then for all $x,y,z\in P$
\[
x\odot y\le z\text{ implies }x\le y\hookrightarrow z.
\]
If
\begin{equation}
L\Big(x,U\big(x^+,L(x,y)\big)\Big)\subseteq L(y)\text{ for all }x,y\in P
\end{equation}
then for all $x,y,z\in P$
\[
x\le y\hookrightarrow z\text{ implies }x\odot y\le z.
\]
\end{theorem}

\begin{proof}
Let $a,b,c\in P$. From the proof of Theorem~\ref{th4} we have that $a\odot b\le c$ is equivalent to $L\big(b,U(a,b^+)\big)\subseteq L(c)$ and that $a\le b\hookrightarrow c$ is equivalent to $U\big(b^+,L(b,c)\big)\subseteq U(a)$. If $a\odot b\le c$ and (3) then $L\big(b,U(a,b^+)\big)\subseteq L(c)$ and hence
\begin{align*}
U\big(b^+,L(b,c)\big) & =U\big(b^+,L(b)\cap L(c)\big)\subseteq U\Big(b^+,L(b)\cap L\big(b,U(a,b^+)\big)\Big)= \\
                      & =U\Big(b^+,L\big(b,U(a,b^+)\big)\Big)\subseteq U(a)
\end{align*}
showing $a\le b\hookrightarrow c$. If, conversely, $a\le b\hookrightarrow c$ and (4) then $U\big(b^+,L(b,c)\big)\subseteq U(a)$ and hence
\begin{align*}
L\big(b,U(a,b^+)\big) & = L\big(b,U(a)\cap U(b^+)\big)\subseteq L\Big(b,U\big(b^+,L(b,c)\big)\cap U(b^+)\Big)= \\
                      & =L\Big(b,U\big(b^+,L(b,c)\big)\Big)\subseteq L(c)
\end{align*}
showing $a\odot b\le c$.
\end{proof}

\begin{remark}
Every Boolean poset $(P,\le,{}',0,1)$ satisfies {\rm(3)} and {\rm(4)} since
\begin{align*}
U\Big(x',L\big(x,U(x',y)\big)\Big) & =U\Big(x',LU\big(L(x,x'),L(x,y)\big)\Big)=U\big(x',LUL(x,y)\big)= \\
& =U\big(x',L(x,y)\big)=UL\big(U(x',x),U(x',y)\big)=ULU(x',y)= \\
& =U(x',y)\subseteq U(y), \\
L\Big(x,U\big(x',L(x,y)\big)\Big) & =L\Big(x,UL\big(U(x',x),U(x',y)\big)\Big)=L\big(x,ULU(x',y)\big)= \\
& =L\big(x,U(x',y)\big)=LU\big(L(x,x'),L(x,y)\big)=LUL(x,y)=L(x,y)\subseteq \\
& \subseteq L(y).
\end{align*}
\end{remark}	

The poset from Figure~8 satisfies (3) since
\[
U\Big(x^+,L\big(x,U(x^+,y)\big)\Big)\left\{
\begin{array}{ll}
=L(y)                                & \text{ if }x=1 \\
\subseteq U(x^+)=\{1\}\subseteq U(y) & \text{ otherwise}.
\end{array}
\right.
\]
Recall that a poset $\mathbf P=(P,\le)$ is called {\em modular} if one of the following equivalent conditions is satisfied:
\begin{align*}
U\big(L(x,y),z\big) & =UL\big(x,L(y,z)\big)\text{ for all }x,y,z\in P\text{ with }z\le x, \\
L\big(U(x,y),z\big) & =LU\big(x,U(y,z)\big)\text{ for all }x,y,z\in P\text{ with }x\le z.
\end{align*}

The next result shows that stronger versions of modularity imply adjointness if $\odot$ and $\hookrightarrow$.

\begin{corollary}\label{cor1}
Let $(P,\le,0,1)$ be a complemented poset of finite length. If
\begin{equation}
U\big(L(A,y),C\big)=UL\big(A,U(y,C)\big)\text{ for all }y\in P\text{ and all }A,C\in2^P\setminus\{\emptyset\}\text{ with }C\le A
\end{equation}
then for all $x,y,z\in P$
\[
x\odot y\le z\text{ implies }x\le y\hookrightarrow z.
\]
If
\begin{equation}
L\big(U(A,B),z\big)=LU\big(A,L(B,z)\big)\text{ for all }z\in P\text{ and all }A,B\in2^P\setminus\{\emptyset\}\text{ with }A\le z
\end{equation}
then for all $x,y,z\in P$
\[
x\odot y\le z\text{ implies }x\le y\hookrightarrow z.
\]
\end{corollary}

\begin{proof}
We have that (5) implies
\begin{align*}
U\Big(x^+,L\big(x,U(x^+,y)\big)\Big) & =U\Big(L\big(U(x^+,y),x\big),x^+\Big)=UL\big(U(x^+,y),U(x,x^+)\big)= \\
                                     & =ULU(x^+,y)=U(x^+,y)\subseteq U(y)
\end{align*}
for all $x,y\in P$, i.e.\ (3). Moreover, (6) implies
\begin{align*}
L\Big(x,U\big(x^+,L(x,y)\big)\Big) & =L\Big(U\big(L(x,y),x^+\big),x\Big)=LU\big(L(x,y),L(x^+,x)\big)=LUL(x,y)= \\
                                   & =L(x,y)\subseteq L(y)	
\end{align*}
for all $x,y\in P$, i.e.\ (4). Now	Corollary~\ref{cor1} follows from Theorem~\ref{th4}.
\end{proof}

\section{Orthogonality, Dedekind-MacNeille completion and convex subsets}

In this section we investigate the binary relation $\perp$ of orthogonality introduced in Section~2 within the Dedekind-MacNeille completion (see e.g.\ \cite B) of a poset and within the poset of non-empty convex subsets of a bounded poset.

Let $\mathbf P=(P,\le)$ be an arbitrary poset. Put
\[
D(\mathbf P):=\{L(A)\mid A\subseteq P\}=\{B\subseteq P\mid LU(B)=B\}.
\]
Then $\mathbf D(\mathbf P):=(D(\mathbf P),\subseteq)$ is a complete lattice, called the {\em Dedekind-MacNeille completion} of $\mathbf P$. It is well-known that in $\mathbf D(\mathbf P)$
\begin{align*}
  \bigvee_{i\in I}A_i & =LU\left(\bigcup_{i\in I}A_i\right), \\
\bigwedge_{i\in I}A_i & =\bigcap_{i\in I}A_i.	
\end{align*}
Moreover, it is well-known that if $a,a_i\in P$ for all $i\in I$ then
\begin{align*}
  L(a)=\bigvee_{i\in I}L(a_i) & \text{ if and only if }a=\bigvee_{i\in I}a_i, \\	
L(a)=\bigwedge_{i\in I}L(a_i) & \text{ if and only if }a=\bigwedge_{i\in I}a_i.
\end{align*}
Further, the mapping $x\mapsto L(x)$ is an embedding of $\mathbf P$ into $\mathbf D(\mathbf P)$, i.e.\ for all $x,y\in P$, $x\le y$ if and only if $L(x)\subseteq L(y)$. If $\mathbf P$ has a smallest element $0$ and a greatest element $1$ then $L(0)$ and $L(1)$ are the smallest, respectively greatest element of $\mathbf D(\mathbf P)$. Hence, if $\mathbf P$ is bounded and $a,b\in P$ then $a\perp_\mathbf Pb$ if and only if $L(a)\perp_{\mathbf D(\mathbf P)}L(b)$.

We can characterize the orthogonality of two elements of the Dedekind-MacNeille completion of a bounded poset as follows.

\begin{theorem}
Let $\mathbf P=(P,\le,0,1)$ be a bounded poset and $A,B\in D(\mathbf P)$. Then the following are equivalent:
\begin{enumerate}[{\rm(i)}]
\item $A\perp_{\mathbf D(\mathbf P)}B$,
\item $\bigvee_\mathbf P(A\cup B)=1$ and $x\wedge y=0$ for all $x\in A$ and all $y\in B$.
\end{enumerate}		
\end{theorem}

\begin{proof}
We have	$A\perp_{\mathbf D(\mathbf P)}B$ if and only if $A\vee_{\mathbf D(\mathbf P)}B=L(1)$ and $A\wedge_{\mathbf D(\mathbf P)}B=L(0)$. Now the following are equivalent: $A\vee_{\mathbf D(\mathbf P)}B=L(1)$, $LU(A,B)=L(1)$, $U(A,B)=U(1)$, $\bigvee_\mathbf P(A\cup B)=1$. Moreover $A\wedge_{\mathbf D(\mathbf P)}B=L(0)$ if and only if $A\cap B=\{0\}$. Now assume $A\cap B=\{0\}$. If $a\in A$, $b\in B$, $c\in P$ and $c\le a,b$ then $a\in A=LU(A)$ and $b\in B=LU(B)$. Hence $c\in LU(A)\cap LU(B)=A\cap B=\{0\}$, i.e. $c=0$ showing $a\wedge b=0$. If, conversely $x\wedge y=0$ for all $x\in A$ and all $y\in B$ and $d\in A\cap B$ then $d=d\wedge d=0$ showing $A\cap B=\{0\}$. So we finally obtain that $A\cap B=\{0\}$ is equivalent to the fact that $x\wedge y=0$ for all $x\in A$ and all $y\in B$.
\end{proof}

Let $\mathbf P=(P,\le)$ be a poset, $a,b,c\in P$ and $A,B\subseteq P$. Recall that $A$ is called {\em convex} if $x,z\in A$, $y\in P$ and $x\le y\le z$ together imply $z\in A$. Let $\Conv\mathbf P$ denote the set of all convex subsets of $\mathbf P$. We define a binary relation $\sqsubseteq$ on $2^P$ by
\[
A\sqsubseteq B\text{ if and only if }A\le_1B\text{ and }A\le_2B.
\]
Since the intersection of convex subsets of $\mathbf P$ is again a convex subset of $\mathbf P$, there exists a smallest (with respect to inclusion) convex subset $\overline A$ of $\mathbf P$ including $A$, the so-called {\em convex hull} of $A$.

\begin{lemma}
Let $(P,\le)$ be a poset and $A\subseteq P$. Then
\[
\overline A=\{x\in P\mid\text{there exist }y,z\in A\text{ with }y\le x\le z\}.
\]
\end{lemma}

\begin{proof}
Let $a,c\in\overline A$ and $b\in P$ and assume $a\le b\le c$. Then there exist $d,e\in A$ with $d\le a$ and $c\le e$. Now $d\le b\le e$ which shows $b\in\overline A$. We have proved $\overline A\in\Conv\mathbf P$. Now let $B\in\Conv\mathbf P$ with $A\subseteq B$ and let $f\in\overline A$. Then there exist $g,h\in A$ with $g\le f\le h$. Since $g,h\in B\in\Conv\mathbf P$ we conclude $f\in B$. This shows $\overline A\subseteq B$ completing the proof of the lemma.
\end{proof}

Observe that $\overline A=A^\downarrow\cap A^\uparrow$. Moreover, let $\Convs\mathbf P$ denote the set of all non-empty elements of $\Conv\mathbf P$.

\begin{lemma}\label{lem2}
Let $\mathbf P=(P,\le)$ be a poset. Then the following holds:
\begin{enumerate}[{\rm(i)}]
\item $\BConv\mathbf P:=(\Conv\mathbf P,\sqsubseteq)$ is a poset,
\item if $\mathbf P$ has a smallest element $0$ and a greatest element $1$ then
\[
\BConvs\mathbf P:=(\Convs\mathbf P,\sqsubseteq,\{0\},\{1\})
\]
is a bounded poset.
\end{enumerate}
\end{lemma}

\begin{proof}
\
\begin{enumerate}[(i)]
\item Since $\le_1$ and $\le_2$ are quasiorders, the same is true for $\sqsubseteq$. Let $A,B\in\Conv\mathbf P$, assume $A\sqsubseteq B$ and $B\sqsubseteq A$ and let $a\in A$. Because of $B\le_2A$ there exists some $b\in B$ with $b\le a$, and because of $A\le_1B$ there exists some $c\in B$ with $a\le c$. Since $b\le a\le c$ and $b,c\in B\in\Conv\mathbf P$ we conclude $a\in B$. This shows $A\subseteq B$. From symmetry reasons we have also $B\subseteq A$, i.e.\ $A=B$.
\item According to (i), $(\Convs\mathbf P,\sqsubseteq)$ is a poset. It is easy to see that $\{0\}\sqsubseteq C\sqsubseteq\{1\}$ for all $C\in\Convs\mathbf P$.
\end{enumerate}
\end{proof}

\begin{example}
For the bounded poset $\mathbf P$ depicted in Fig.~9

\vspace*{-2mm}

\begin{center}
\setlength{\unitlength}{7mm}
\begin{picture}(6,6)
\put(3,1){\circle*{.3}}
\put(1,3){\circle*{.3}}
\put(5,3){\circle*{.3}}
\put(3,5){\circle*{.3}}
\put(3,1){\line(-1,1)2}
\put(3,1){\line(1,1)2}
\put(3,5){\line(-1,-1)2}
\put(3,5){\line(1,-1)2}
\put(2.85,.3){$0$}
\put(.35,2.85){$a$}
\put(5.4,2.85){$b$}
\put(2.85,5.4){$1$}
\put(-.5,-.75){{\rm Figure~9. Bounded poset $\mathbf P$}}
\end{picture}
\end{center}

\vspace*{5mm}

the bounded poset $\BConvs\mathbf P$ is visualized in Fig.~10

\vspace*{-2mm}

\begin{center}
\setlength{\unitlength}{7mm}
\begin{picture}(14,14)
\put(7,1){\circle*{.3}}
\put(5,3){\circle*{.3}}
\put(9,3){\circle*{.3}}
\put(7,5){\circle*{.3}}
\put(1,7){\circle*{.3}}
\put(5,7){\circle*{.3}}
\put(9,7){\circle*{.3}}
\put(13,7){\circle*{.3}}
\put(7,9){\circle*{.3}}
\put(5,11){\circle*{.3}}
\put(9,11){\circle*{.3}}
\put(7,13){\circle*{.3}}
\put(7,1){\line(-1,1)6}
\put(7,1){\line(1,1)6}
\put(7,13){\line(-1,-1)6}
\put(7,13){\line(1,-1)6}
\put(5,7){\line(1,-1)4}
\put(5,7){\line(1,1)4}
\put(9,7){\line(-1,-1)4}
\put(9,7){\line(-1,1)4}
\put(6.85,.3){$0$}
\put(4.05,2.85){$0a$}
\put(4.05,6.85){$ab$}
\put(4.05,10.85){$a1$}
\put(.35,6.85){$a$}
\put(9.4,2.85){$0b$}
\put(9.4,6.85){$0ab1$}
\put(9.4,10.85){$b1$}
\put(13.4,6.85){$b$}
\put(7.4,4.85){$0ab$}
\put(7.4,8.85){$ab1$}
\put(6.85,13.4){$1$}
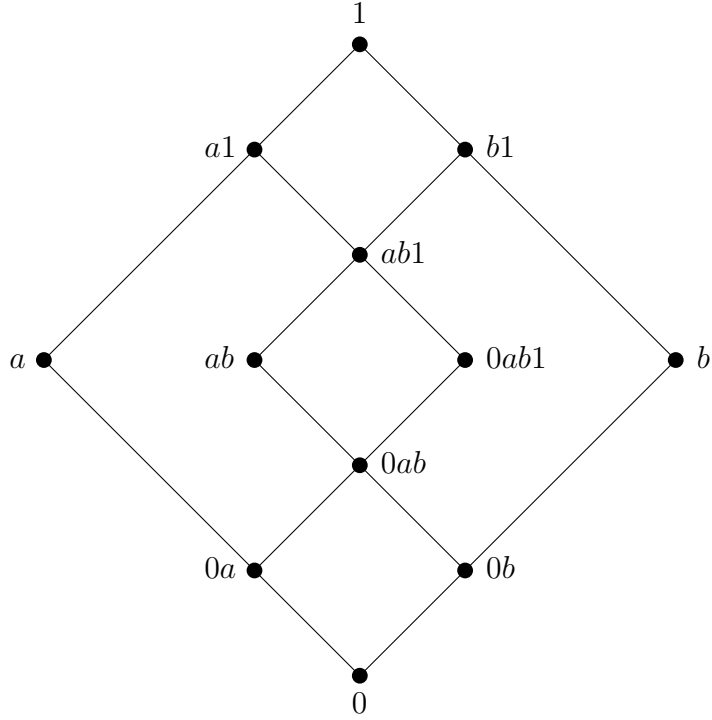
\put(2,-.75){{\rm Figure~10. Poset of convex subsets of $\mathbf P$}}
\end{picture}
\end{center}

\vspace*{4mm}

Here we write $0ab1$ instead of $\{0,a,b,1\}$, and so on.
\end{example}

Let $\mathbf P=(P,\le,0,1)$ be a bounded poset and $A,B\subseteq P$. Then we define
\[
A\perp_\mathbf PB\text{ if and only if }x\perp_\mathbf Py\text{ for all }x\in A\text{ and all }y\in B.
\]

Finally, we show that two non-empty subsets of a bounded poset $(P,\le,0,1)$ are orthogonal if and only if their convex hulls are orthogonal in the bounded poset of non-empty convex subsets of $P$.

\begin{theorem}
Let $\mathbf P=(P,\le,0,1)$ be a bounded poset and $A,B$ non-empty subsets of $P$. Then $A\perp_\mathbf PB$ if and only if $\overline A\perp_{\BConvs\mathbf P}\overline B$.
\end{theorem}

\begin{proof}
First assume $A\perp_\mathbf PB$. Let $C\in\Convs\mathbf P$ with $\overline A,\overline B\sqsubseteq C$. Assume $a\in C$. Then there exist $b\in\overline A$ and $c\in\overline B$ with $b,c\le a$. Hence there exist $d\in A$ and $e\in B$ with $d\le b$ and $e\le c$. Since $d,e\le a$ and $d\vee e=1$ we conclude $a=1$. Hence $\overline A\vee_{\Convs\mathbf P}\overline B=\{1\}$. Now let $D\in\Convs\mathbf P$ with $D\sqsubseteq\overline A,\overline B$. Assume $f\in D$. Then there exist $g\in\overline A$ and $h\in\overline B$ with $f\le g,h$. Hence there exist $i\in A$ and $j\in B$ with $g\le i$ and $h\le j$. Since $f\le i,j$ and $i\wedge j=0$ we conclude $f=0$. Hence $\overline A\wedge_{\Convs\mathbf P}\overline B=\{0\}$. This shows $\overline A\perp_{\BConvs\mathbf P}\overline B$. Conversely, assume $\overline A\perp_{\BConvs\mathbf P}\overline B$. Let $k\in A$, $l\in B$ and $m,n\in P$ with $m\le k,l\le n$. Then $\downarrow m\sqsubseteq\overline A,\overline B\sqsubseteq\uparrow n$ and hence $\downarrow m=\{0\}$ and $\uparrow n=\{1\}$ whence $m=0$ and $n=1$. This shows $k\perp_\mathbf Pl$. Since $k$ and $l$ were arbitrary elements of $A$ and $B$, respectively, we get $A\perp_\mathbf PB$.
\end{proof}








Authors' addresses:

Michal Botur \\
Palack\'y University Olomouc \\
Faculty of Science \\
Department of Algebra and Geometry \\
17.\ listopadu 12 \\
771 46 Olomouc \\
Czech Republic \\
michal.botur@upol.cz

Ivan Chajda \\
Palack\'y University Olomouc \\
Faculty of Science \\
Department of Algebra and Geometry \\
17.\ listopadu 12 \\
771 46 Olomouc \\
Czech Republic \\
ivan.chajda@upol.cz

Helmut L\"anger \\
TU Wien \\
Faculty of Mathematics and Geoinformation \\
Institute of Discrete Mathematics and Geometry \\
Wiedner Hauptstra\ss e 8-10 \\
1040 Vienna \\
Austria, and \\
Palack\'y University Olomouc \\
Faculty of Science \\
Department of Algebra and Geometry \\
17.\ listopadu 12 \\
771 46 Olomouc \\
Czech Republic \\
helmut.laenger@tuwien.ac.at

\begin{thebibliography}9
\bibitem B
G.~Birkhoff, Lattice Theory. AMS, Providence; RI, 1979. ISBN 0-8218-1025-1.
\bibitem{CL19}
I.~Chajda and H.~L\"anger, Residuation in modular lattices and posets. Asian-Eur.\ J.\ Math.\ {\bf12} (2019), 1950092 (10 pp.).
\bibitem{CL25a}
I.~Chajda and H.~L\"anger, Operators on complemented lattices. Soft Computing {\bf29} (2025), 3115--3123.
\bibitem{CL25b}
I.~Chajda and H.~L\"anger, Algebraic structures formalizing the logic with unsharp implication and negation. Logic J.\ IGPL {\bf33} (2025), 36--48.
\bibitem D
R.~P.~Dilworth, On complemented lattices, T\^ohoku Math.\ J.\ {\bf47} (1940), 18--23.
\bibitem N
J.~Niederle, Boolean and distributive ordered sets: characterization and representation by sets. Order {\bf12} (1995), 189--210.
\bibitem P
J.~Paseka, Note on distributive posets, Math.\ Appl.\ {\bf1} (2012), 197--206.
\end{thebibliography}
\end{document}